\documentclass{article}

\usepackage{amsmath,amsthm,amssymb,amsfonts}
\usepackage{verbatim}
\usepackage{graphicx}
\usepackage{stmaryrd} %for \leftrightarroweq
\usepackage{hyperref}

\usepackage[colorinlistoftodos]{todonotes}

\theoremstyle{plain}
\newtheorem{thm}{Theorem}

\newtheorem{prop}[thm]{Proposition}
\newtheorem{cor}[thm]{Corollary}

\theoremstyle{definition}
\newtheorem{definition}[thm]{Definition}
\newtheorem{exl}[thm]{Example}

\numberwithin{thm}{section}

\def\Z{{\mathbb Z}}

\begin{document}
\title{Digital Shy Maps}
\author{Laurence Boxer
         \thanks{
    Department of Computer and Information Sciences,
    Niagara University,
    Niagara University, NY 14109, USA;
    and Department of Computer Science and Engineering,
    State University of New York at Buffalo.
    E-mail: boxer@niagara.edu
    }
}

\date{ }
\maketitle

\begin{abstract}
We study properties of shy maps in digital topology. \\

Key words and phrases: digital image, continuous multivalued function, shy map, isomorphism, Cartesian product, wedge
\end{abstract}

\section{Introduction}
Continuous functions between digital images were
introduced in ~\cite{Rosenfeld} and have been
explored in many subsequent papers. A {\em shy} map is a continuous
function with certain additional restrictions (see Definition~\ref{shyDef}). 
Shy maps were studied in~\cite{Boxer05,Boxer14,BoxS16}.
In the current paper, we develop additional properties of shy maps.

\section{Preliminaries}
\label{prelims}
Let $\Z$ be the set of integers. A digital image is a pair
$(X,\kappa)$ where $X \subset \Z^n)$ for some positive integer $n$,
and $\kappa$ is an adjacency relation for $Z^n$. Thus, a digital
image is a graph in which $X$ is the set of
vertices and edges correspond to $\kappa$-adjacent points of $X$.

Much of the exposition in this section is quoted or paraphrased from the papers referenced.

\subsection{Digitally continuous functions}

We will assume familiarity with the topological theory of digital images. See, e.g., \cite{Boxer94} for the standard definitions. All digital images $X$ are assumed to carry their own adjacency relations (which may differ from one image to another). We write the image as $(X,\kappa)$, where $\kappa$ 
represents the adjacency relation, when we wish to emphasize  or clarify the  adjacency relation.

Among the commonly used adjacencies are the $c_u$-adjacencies.
Let $x,y \in \Z^n$, $x \neq y$. Let $u$ be an integer,
$1 \leq u \leq n$. We say $x$ and $y$ are $c_u$-adjacent if
\begin{itemize}
\item There are at most $u$ indices $i$ for which 
      $|x_i - y_i| = 1$.
\item For all indices $j$ such that $|x_j - y_j| \neq 1$ we
      have $x_j=y_j$.
\end{itemize}
We often label a $c_u$-adjacency by the number of points
adjacent to a given point in $\Z^n$ using this adjacency.
E.g.,
\begin{itemize}
\item In $\Z^1$, $c_1$-adjacency is 2-adjacency.
\item In $\Z^2$, $c_1$-adjacency is 4-adjacency and
      $c_2$-adjacency is 8-adjacency.
\item In $\Z^3$, $c_1$-adjacency is 6-adjacency,
      $c_2$-adjacency is 18-adjacency, and $c_3$-adjacency
      is 26-adjacency.
\end{itemize}

A subset $Y$ of a digital image $(X,\kappa)$ is
{\em $\kappa$-connected}~\cite{Rosenfeld},
or {\em connected} when $\kappa$
is understood, if for every pair of points $a,b \in Y$ there
exists a sequence $P=\{y_i\}_{i=0}^m \subset Y$ such that
$a=y_0$, $b=y_m$, and $y_i$ and $y_{i+1}$ are 
$\kappa$-adjacent for $0 \leq i < m$. The set $P$ is called a {\em path}
from $a$ to $b$. 
The following generalizes a definition of
~\cite{Rosenfeld}.

\begin{definition}
\label{continuous}
{\rm ~\cite{Boxer99}}
Let $(X,\kappa)$ and $(Y,\lambda)$ be digital images. A function
$f: X \rightarrow Y$ is $(\kappa,\lambda)$-continuous if for
every $\kappa$-connected $A \subset X$ we have that
$f(A)$ is a $\lambda$-connected subset of $Y$. \qed
\end{definition}

When the adjacency relations are understood, we will simply say that $f$ is \emph{continuous}.

Given positive integers $u,v$ such that $u \leq v$ a digital interval
is a set of the form
\[ [u,v]_{\Z} = \{z \in \Z \, | \, u \leq z \leq v\}
\]
treated as a digital image with the $c_1$-adjacency.

The term {\em path from $a$ to $b$} is also used for a continuous function
$p: ([0,m]_{\Z}, c_1) \to (Y,\kappa)$ such that $f(0)=a$ and $f(m)=b$. Context
generally clarifies which meaning of {\em path} is appropriate.

Continuity can be reformulated in terms of adjacency of points:
\begin{thm}
\label{pt-cont}
{\rm ~\cite{Rosenfeld,Boxer99}}
A function $f:X\to Y$ is continuous if and only if, for any adjacent points $x,x'\in X$, the points $f(x)$ and $f(x')$ are equal or adjacent. \qed
\end{thm}

\begin{thm}
\label{composition}
\rm{\cite{Boxer94,Boxer99}}
Let $f: (A, \kappa) \to (B, \lambda)$ and
$g: (B, \lambda) \to (C, \mu)$ be continuous functions between
digital images.  Then $g \circ f: (A, \kappa) \to (C, \mu)$ is
continuous. \qed
\end{thm}

\begin{definition}
\label{iso}
A function $f:X\to Y$ is an {\em isomorphism}~\cite{Boxer06}
(called a {\em homeomorphism} in~\cite{Boxer94}) if $f$ is
a continuous bijection and $f^{-1}$ is continuous. \qed
\end{definition}

\begin{definition}
\label{shyDef}
\rm{~\cite{Boxer05}}
Let $f: (X, \kappa) \to (Y, \lambda)$ be a continuous surjection of
digital images. We say $f$ is {\em shy} if
\begin{itemize}
\item for every $y \in Y$, $f^{-1}(y)$ is $\kappa$-connected; and
\item for every $\lambda$-adjacent $y_0,y_1 \in Y$, 
      $f^{-1}(\{y_0,y_1\})$ is $\kappa$-connected. \qed
\end{itemize}
\end{definition}

It is known~\cite{Boxer05} that shy maps induce surjections of 
fundamental groups and that~\cite{BoxS16} a continuous surjection
$f: (X, \kappa) \to (Y, \lambda)$ is shy if and only if
$f^{-1}: (Y, \lambda) \multimap (X, \kappa)$ is a connectivity
preserving multivalued function (see Definition~\ref{preserving}).
We also have the following.

\begin{thm}
\rm{\cite{Boxer14}}
\label{shy-cont}
Let $f: X \to Y$ be a continuous surjection of digital
images. Then $f$ is shy if and only if
for every connected subset $Y'$ of $Y$, $f^{-1}(Y')$ is connected.
\end{thm}

\subsection{Normal product adjacency}
The {\em normal product adjacency} has been used in many
papers for Cartesian products of graphs. 

\begin{definition}
\label{normal}
\cite{Berge}
Let $(X,\kappa)$ and $(Y,\lambda)$ be digital images. The {\em normal
product adjacency} $k_*(\kappa,\lambda)$ for $X \times Y$ is defined
as follows. Two members $(x_0,y_0)$ and $(x_1, y_1)$ of $X \times Y$
are $k_*(\kappa,\lambda)$-adjacent if and only if one of the following
is valid.
\begin{itemize}
\item $x_0=x_1$, and $y_0$ and $y_1$ are $\lambda$-adjacent; or
\item $x_0$ and $x_1$ are $\kappa$-adjacent, and $y_0=y_1$; or
\item $x_0$ and $x_1$ are $\kappa$-adjacent, and
      $y_0$ and $y_1$ are $\lambda$-adjacent. \qed
\end{itemize}
\end{definition}

We will use the following properties.

\begin{prop}
\rm{\cite{Han05}}
\label{projections}
The projection maps 
$p_1: (X \times Y, k_*(\kappa,\lambda)) \to (X, \kappa)$ and
$p_2: (X \times Y, k_*(\kappa,\lambda)) \to (Y, \lambda)$
defined by $p_1(x,y)=x$, $p_2(x,y)=y$ are
$(k_*(\kappa,\lambda), \kappa)$-continuous and
$(k_*(\kappa,\lambda), \lambda)$-continuous, respectively. \qed
\end{prop}

\begin{prop}
\label{prodAndc_u}
\rm{~\cite{BoxKar12}}
In $\Z^{m+n}$, $k_*(c_m,c_n) = c_{m+n}$; i.e., given points
$x,x' \in \Z^m$ and $y,y' \in \Z^n$, $(x,y)$ and $(x',y')$ are 
$k_*(c_m,c_n)$-adjacent in $\Z^{m+n}$ if and
only if they are $c_{m+n}$-adjacent. \qed
\end{prop}

\subsection{Digital multivalued functions}
To ameliorate limitations and anomalies that appear in the study of
continuous functions between digital images, several authors
have considered multivalued functions with various forms of
continuity. Functions with
{\em weak continuity} and {\em strong continuity} were
introduced in~\cite{Ts} and studied further in~\cite{BoxS16}.
{\em Connectivity preserving multivalued functions} were
introduced in ~\cite{Kov} and studied further in~\cite{BoxS16}.
{\em Continuous multivalued functions} were introduced in~\cite{egs08,egs12} and studied further 
in~\cite{gs15,Boxer14,BoxS16}. We use the following.

\begin{definition}
\rm{\cite{Kov}}
\label{preserving}
A multivalued function $f: X \multimap Y$ is 
{\em connectivity preserving} if for every connected subset $A$ of
$X$, $f(A)$ is a connected subset of $Y$. \qed
\end{definition}

\begin{definition}
\label{adjSets}
Let $A$ and $B$ be subsets of a digital image $(X,\kappa)$. We
say $A$ and $B$ are {\em $\kappa$-adjacent}, or {\em adjacent} 
for short, if there exist $a \in A$ and $b \in B$ such that either
$a=b$ or $a$ and $b$ are $\kappa$-adjacent. \qed
\end{definition}

\begin{definition}
\rm{\cite{Ts}}
\label{weakCont}
A multivalued function $f: X \multimap Y$ has 
{\em weak continuity} if for every pair $x,y$ of
adjacent points in $X$, the sets $f(x)$ and $f(y)$ are
adjacent in $Y$. \qed
\end{definition}

\begin{prop}
\label{weakCharacterize}
\rm{\cite{BoxS16}}
Let $f: X \multimap Y$ be a multivalued function between
digital images. Then $f$ is connectivity preserving if
and only if $f$ has weak continuity and for all $x \in X$,
$f(x)$ is a connected subset of $Y$.
\qed
\end{prop}

\begin{thm}
\label{shy-connPreserving}
\rm{\cite{BoxS16}}
A continuous surjection $f: X \to Y$ of digital images is shy
if and only if $f^{-1}: Y \multimap X$ is a connectivity 
preserving multivalued function. \qed
\end{thm}

\section{Shy, constant, and isomorphism functions}
\label{moreProps}
\begin{prop}
A constant map between connected digital images is shy. I.e., if
$X$ is a connected digital image and $y \in Z^n$, then the function
$f: X \to \{y\}$ is shy.
\end{prop}

\begin{proof} This is obvious.
\end{proof}

Previous results give the following characterizations of shy maps.

\begin{thm}
Let $f: X \to Y$ be a continuous surjection between
digital images. The following are equivalent.
\begin{enumerate}
\item $f$ is a shy map.
\item For every connected $Y' \subset Y$, $f^{-1}(Y')$ is a
      connected subset of $X$.
\item $f^{-1}: Y \multimap X$ is a connectivity preserving
      multi-valued function.
\item $f^{-1}: Y \multimap X$ is a multi-valued function with
      weak continuity such that for all $y \in Y$,
      $f^{-1}(y)$ is a connected subset of $X$.
\end{enumerate}
\end{thm}

\begin{proof}
The equivalence of the first three statements follows from Theorem~\ref{shy-cont} and
Theorem~\ref{shy-connPreserving}. The equivalence of the
third and fourth statements follows from Proposition~\ref{weakCharacterize}.
\end{proof}

\begin{thm}
Let $f: (X, \kappa) \to (Y, \lambda)$ be a continuous surjection.
\begin{itemize}
\item If $f$ is an isomorphism then $f$ is shy.
\item If $f$ is shy and one-to-one, then $f$ is an isomorphism.
\end{itemize}
\end{thm}

\begin{proof}
Since a single point is connected, the first assertion follows
easily from Definitions~\ref{iso} and~\ref{shyDef}.

Suppose $f$ is shy and one-to-one. Shyness implies that $f$ is a surjection,
hence a bijection; and, from Definitions~\ref{continuous} and 
~\ref{shyDef}, that $f^{-1}$ is continuous. Therefore,
$f$ is an isomorphism.
\end{proof}

\section{Operations that preserve shyness}
\label{ops}
We show that composition preserves shyness.

\begin{thm}
Let $f: A \to B$  and $g: B \to C$ be shy maps. Then
$g \circ f: A \to C$ is shy.
\end{thm}

\begin{proof}
Let $C'$ be a connected subset of $C$. By Theorem~\ref{shy-cont},
$g^{-1}(C')$ is a connected subset of $B$. Therefore, by
Theorem~\ref{shy-cont}, $(g \circ f)^{-1}(C')=f^{-1}(g^{-1}(C'))$
is a connected subset of $A$. The assertion follows from
Theorem~\ref{shy-cont}.
\end{proof}

We will need the following.

\begin{prop}
\label{productMap}
Let $f: (A,\alpha) \to (C,\gamma)$  and $g: (B,\beta) \to (D,\delta)$.
Then $f$ and $g$ are continuous if and only if the function
$f \times g: (A \times B, k_*(\alpha,\beta)) \to (C \times D, k_*(\gamma,\delta))$ defined by $(f \times g)(a,b) = (f(a), g(b))$ is 
continuous.
\end{prop}

\begin{proof}
Suppose $f$ and $g$ are continuous.
Let $(a,b)$ and $(a',b')$ be $k_*(\alpha,\beta)$-adjacent. Then
$a$ and $a'$ are equal or $\alpha$-adjacent, so
$f(a)$ and $f(a')$ are equal or $\gamma$-adjacent; and
$b$ and $b'$ are equal or $\beta$-adjacent, so
$g(b)$ and $g(b')$ are equal or $\delta$-adjacent. It follows
from Definition~\ref{normal} that
$(f \times g)(a,b)$ and $(f \times g)(a',b')$ are equal or
$k_*(\gamma,\delta)$-adjacent.  Therefore, $f \times g$ is
continuous.

Conversely, suppose $f \times g$ is continuous. From
Proposition~\ref{projections}, we know the projection maps 
$p_1: (X \times Y, k_*(\kappa,\lambda)) \to (X, \kappa)$ and
$p_2: (X \times Y, k_*(\kappa,\lambda)) \to (Y, \lambda)$
defined by $p_1(x,y)=x$, $p_2(x,y)=y$, are
continuous. It follows from Theorem~\ref{composition}
that $f = p_1 \circ (f \times g)$ and $g=p_2 \circ (f \times g)$ are continuous.
\end{proof}

\begin{prop}
\label{productSurj}
Let $f: A \to C$  and $g: B \to D$
be functions. Then the function
$f \times g: A \times B \to C \times D$ defined by $(f \times g)(a,b) = (f(a), g(b))$ is 
a surjection if and only if $f$ and $g$ are surjections.
\end{prop}

\begin{proof}
Let $(c,d) \in C \times D$. If $f$ and $g$ are surjections,
there are $a \in A$ and $b \in B$ such that
$f(a)=c$ and $g(b)=d$. Therefore,
$(f \times g)(a,b)=(c,d)$. Thus, $f \times g$ is a surjection.

Conversely, if $f \times g$ is a surjection, it follows easily that
$f$ and $g$ are surjections.
\end{proof}

Cartesian products preserve shyness with respect to the normal product
adjacency, as shown in the following.

\begin{thm}
\label{products}
Let $f: (A,\alpha) \to (C,\gamma)$  and $g: (B,\beta) \to (D,\delta)$
be continuous surjections. Then $f$ and $g$ are shy maps if and only if
the function $f \times g: (A \times B, k_*(\alpha,\beta)) \to (C \times D, k_*(\gamma,\delta))$ is a shy map.
\end{thm}

\begin{proof}
Suppose $f$ and $g$ are shy. Then they are surjections. By 
Propositions~\ref{productMap} and~\ref{productSurj}, $f \times g$ is a 
continuous surjection.

Let $(c,d) \in C \times D$ and let 
$(a,b), (a',b') \in (f \times g)^{-1}(c,d)$. Since
$f^{-1}(c)$ is connected, there is a path $P$ in
$f^{-1}(c)$ from $a$ to $a'$. Therefore,
$P \times \{b\}$ is a path in
$f^{-1}(c) \times \{b\} \subset (f \times g)^{-1}(c,d)$ from
$(a,b)$ to $(a',b)$.  Since
$g^{-1}(d)$ is connected, there is a path $Q$ in
$g^{-1}(d)$ from $b$ to $b'$. Therefore,
$\{a'\}\times Q$ is a path in
$\{a'\}\times g^{-1}(d) \subset (f \times g)^{-1}(c,d)$ from
$(a',b)$ to $(a',b')$. Thus,
$(P \times \{b\}) \cup (\{a'\}\times Q)$ is a path in
$(f \times g)^{-1}(c,d)$ from $(a,b)$ to $(a',b')$. Therefore,
$(f \times g)^{-1}(c,d)$ is connected.

Let $(c,d)$ and $(c',d')$ be $k_*(\gamma,\delta)$-adjacent in
$C \times D$. Then $c$ and $c'$ are equal or $\alpha$-adjacent,
and $d$ and $d'$ are equal or $\beta$-adjacent. Let
$\{(a,b), (a',b')\} \subset (f \times g)^{-1}(\{(c,d), (c',d')\}$.
Since $f$ is shy,
$f^{-1}(\{c,c'\})$ is connected, so there is a path $P$ in
$f^{-1}(\{c,c'\})$ from $a$ to $a'$. Similarly, there is a path
$Q$ in $g^{-1}(\{d,d'\})$ from $b'$ to $b$. Thus,
$(P \times \{b\}) \cup (\{a'\} \times Q)$ is a path in
\[ (f^{-1}(\{c,c'\}) \times \{b\}) \cup 
 (\{a'\} \times g^{-1}(\{d,d'\}) \subset
 (f \times g)^{-1}(\{(c,d), (c',d')\})
 \]
from $(a,b)$ to $(a',b')$. Therefore,
$(f \times g)^{-1}(\{(c,d), (c',d')\})$ is connected.

It follows from Definition~\ref{shyDef} that $f \times g$ is a
shy map.

Conversely, suppose $f \times g$ is a shy map. It follows from
Propositions~\ref{productMap} and~\ref{productSurj}
that $f$ and $g$ are continuous surjections.

Let $U$ be a connected subset of $C$. Then for $d \in D$, $U \times \{d\}$
is a connected subset of $(C \times D, k_*(\gamma,\delta))$. The shyness
of $f \times g$ implies 
\[ f^{-1}(U) \times g^{-1}(\{d\})=(f \times g)^{-1}(U \times \{d\}) 
\]
is connected in $(A \times B, k_*(\alpha,\beta))$. From
Proposition~\ref{projections}, it follows that
\[ f^{-1}(U) = p_1(f^{-1}(U) \times g^{-1}(\{d\}))
\]
is connected in $A$. Thus, $f$ is shy. A similar argument shows $g$ is shy.
\end{proof}

\begin{cor}
Let $A \subset \Z^m$, $B \subset \Z^n$, $C \subset \Z^u$,
$D \subset \Z^v$. Suppose $f: (A,c_m) \to (C,c_u)$  and 
$g: (B,c_n) \to (D,c_v)$ are functions. Then the function
$f \times g: (A \times B, c_{m+n}) \to (C \times D, c_{u+v})$ 
is a shy map if and only if $f$ and $g$ are shy maps.
\end{cor}

\begin{proof} By Proposition~\ref{prodAndc_u},
in $\Z^{m+n}$, $k_*(c_m,c_n) = c_{m+n}$ and in
$\Z^{u+v}$, $k_*(c_u,c_v) = c_{u+v}$. The assertion
follows from Theorem~\ref{products}.
\end{proof}

The digital image $(C,\kappa)$ is the {\em wedge} of its subsets 
$X$ and $Y$, denoted $C = X \wedge Y$, if $C = X \cup Y$ and
$X \cap Y = \{x_0\}$ for some point $x_0$, such that if $x \in X$, $y \in Y$,
and $x$ and $y$ are $\kappa$-adjacent, then $x_0 \in \{x,y\}$.
Let $f: A \to C$ and
$g: B \to D$ be functions between digital images, with
$A \cap B = \{x_0\}$, $C \cap D = \{y_0\}$, and
$f(x_0)=y_0=g(x_0)$. We define
$f \wedge g: A \wedge B \to C \wedge D$ by
\[ (f \wedge g)(x) = \left \{ \begin{array}{ll}
                     f(x) & \mbox{if } x \in A; \\
                     g(x) & \mbox{if } x \in B.
                     \end{array} \right.
\]
Since $f(x_0)=y_0=g(x_0)$, $f \wedge g$ is well defined. We have
the following.

\begin{thm}
Let $f: A \to C$ and $g: B \to D$ be functions between digital 
images, with $A \cap B = \{x_0\}$, $C \cap D = \{y_0\}$, and
$f(x_0)=y_0=g(x_0)$. Then $f$ and $g$ are shy maps if and only if
$f \wedge g$ is a shy map.
\end{thm}

\begin{proof} If $f$ and $g$ are shy, it is easy to see that
$f \wedge g$ is continuous and surjective.

Let $U$ be a connected subset of $C \wedge D$. Let 
$v_0,v_1 \in f^{-1}(U)$. The connectedness of $U$ implies there
is a path $P$ in $U$ from $f(v_0)$ to $f(v_1)$, i.e., a
continuous $p: [0,m]_{\Z} \to U$ such that 
$p(0)=f(v_0)$ and $p(m)=f(v_1)$, where $P=p([0,m]_{\Z})$.
\begin{itemize}
\item If $U \subset C$, then $(f \wedge g)^{-1}(P) = f^{-1}(P)$ is
      a connected subset of $A$.
\item If $U \subset D$, then $(f \wedge g)^{-1}(P) = g^{-1}(P)$ is
      a connected subset of $B$.
\item Otherwise, there are integers
      $0 = i_1 < i_2 < \cdots < i_n = m$ such that 
      $p([i_j,i_{j+1}]_{\Z}) \subset C \cap U$ for even $j$,
      $p([i_j,i_{j+1}]_{\Z}) \subset D \cap U$ for odd $j$.
      Therefore, 
      \[ (f\wedge g)^{-1}(P) = \bigcup_{\mbox{even } j}f^{-1}(p([i_j,i_{j+1}]_{\Z})) \cup \bigcup_{\mbox{odd } j}g^{-1}(p([i_j,i_{j+1}]_{\Z})) \]
      is a union of connected sets, each containing $x_0$. Hence
      $(f\wedge g)^{-1}(P)$ is a connected subset of
      $(f\wedge g)^{-1}(U)$ containing $\{v_0,v_1\}$. Thus,
      $(f\wedge g)^{-1}(U)$ is connected.
\end{itemize}
In all cases, $(f\wedge g)^{-1}(U)$ is connected.
We conclude from Theorem~\ref{shy-cont} 
that $f\wedge g$ is a shy map.

Conversely, suppose $f\wedge g$ is a shy map. Then both $f$
and $g$ are continuous surjections. If
$U$ is a connected subset of $C$ and $V$ is a connected subset of $D$, we have by 
Theorem~\ref{shy-cont} that
$f^{-1}(U) = (f\wedge g)^{-1}(U)$ is a connected subset of
$A$ and
$g^{-1}(V) = (f\wedge g)^{-1}(V)$ is a connected subset of
$B$. From Theorem~\ref{shy-cont}, it follows that
$f$ and $g$ are shy maps.
\end{proof}

\section{Shy maps into $\Z$}
\label{intoZ}
\begin{thm}
\label{monotone}
Let $X$ and $Y$ be connected subsets of $(\Z,c_1)$,
and let $f: X \to Y$ be a continuous surjection.
Then $f$ is shy if and only if $f$ is either 
monotone non-decreasing or monotone non-increasing.
\end{thm}

\begin{proof}
Suppose $f$ is shy. If $f$ is not monotone, then either
\begin{equation}
\label{b-max}
\mbox{for some } a,b,c \in X \mbox{ with } a < b < c,~f(a)<f(b)
\mbox{ and } f(b)>f(c),
\end{equation}
or
\begin{equation}
\label{b-min}
\mbox{for some } a,b,c \in X \mbox{ with } a < b < c,~f(a)>f(b)
\mbox{ and } f(b)<f(c).
\end{equation}
In case~(\ref{b-max}), the continuity of $f$ implies there
exist $s,t \in X$ such that $a \leq s < b < t \leq c$ and
$f(s)=f(t)=f(b)-1$. 
Thus, $s,t \in f^{-1}(f(b)-1)$ and 
$b \not \in f^{-1}(f(b)-1)$, so $f^{-1}(f(b)-1)$ is not $c_1$-connected,
a contradiction of the shyness of $f$. Case~(\ref{b-min}) generates
a similar contradiction. Thus, we obtain a contradiction by
assuming that $f$ is not monotone.

Suppose $f$ is monotone. We may assume without loss of generality 
that $f$ is non-decreasing. Let $Y'$ be a connected subset of $Y$ 
and let $x_0,x_1 \in f^{-1}(Y')$. Without loss of generality,
$x_0 < x_1$. Since $f$ is continuous and non-decreasing,
$f([x_0,x_1]_{\Z}) = [f(x_0),f(x_1)]_{\Z}$ is a connected set
containing $\{f(x_0),f(x_1)\} \subset Y'$. Thus,
\[ [x_0,x_1]_{\Z} \subset f^{-1}([f(x_0),f(x_1)]_{\Z}) \subset f^{-1}(Y'),
\]
so $[x_0,x_1]_{\Z}$ is a connected subset of $f^{-1}(Y')$ containing
$\{x_0,x_1\}$. Since $x_0$ and $x_1$ were arbitrarily chosen,
we must have that $f^{-1}(Y')$ is $c_1$-connected. Therefore,
$f^{-1}$ is a connectivity preserving multivalued function.
It follows from Theorem~\ref{shy-cont} that $f$ is shy.
\end{proof}

\begin{thm}
Let $S$ be a digital simple closed curve. Let
$f: S \to Y \subset (\Z,c_1)$ be a shy map. Then either
$Y = \{z\}$ or $Y = \{z,z+1\}$ for some $z \in \Z$.
\end{thm}

\begin{proof}
If $f$ is not a constant function, i.e., if $Y \neq \{z\}$, then
\[ z_0 = \min \{f(x) \, | \, x \in S\} < 
   \max \{f(x) \, | \, x \in S\} = z_1.
\]
Let $x_i \in f^{-1}(z_i)$, $i \in \{0,1\}$. There are two
distinct digital arcs, $A$ and $B$, connecting $x_0$ and $x_1$ 
in~$S$. If $z_1 - z_0 > 1$, then the continuity of $f|_A$ and
$f|_B$ implies there are points $a \in A$ and $b \in B$ such
that $f(a)=f(b)=z_0+1$. Since $a$ and $b$ are in distinct 
components of $S \setminus (f^{-1}(\{z_0,z_1\})$, 
$f^{-1}(z_0+1)$ is disconnected. This is
contrary to the assumption that $f$ is shy. The assertion follows.
\end{proof}

\begin{thm}
\label{disconnections}
Let $(X,\kappa)$ be a digital image and let
$r \in X$ be such that $X \setminus \{r\}$ is
$\kappa$-disconnected. Let $f: (X, \kappa) \to Y \subset (\Z,c_1)$ be
a shy map. Then there are at most 2 components of
$X \setminus \{r\}$ on which $f$ is not equal to the constant 
function with value $f(r)$.
\end{thm}

\begin{proof} 
Suppose $A$ is a component of $X\setminus \{r\}$ on which $f$ 
is not constant.
By continuity of $f$, there exists $a \in A$ such that
$|f(a)-f(r)|=1$. Suppose $B$ is another component of 
$X\setminus \{r\}$ on which $f$ is not constant. Similarly,
there exists $b \in b$ such that $|f(b)-f(r)|=1$. We must have
$f(a) \neq f(b)$, since $f^{-1}(f(a))$ is connected and every
path in $X$ from $a$ to $b$ contains $r$. Therefore, we may
assume $f(a)=f(r)-1$, $f(b)=f(r)+1$.

Suppose $C$ is a component of $X\setminus \{r\}$ that is distinct
from $A$ and $B$. Suppose there exists $c \in C$ such that
$f(c) \neq f(r)$. If $f(c) < f(r)$ then, by continuity of $f$,
there exists $c' \in C$ such that $f(c')=f(r)-1=f(a)$. But this
is impossible, since $f^{-1}(f(a))$ is connected and every
path in $X$ from $a$ to $c'$ contains $r$. Similarly, if we
assume $f(c) > f(r)$ we get a contradiction. The assertion follows.
\end{proof}

\begin{exl}
\label{tree-ex}
Let $T$ be a tree. Let $r$ be the root vertex of $T$.
Let $\{v_i\}_{i=0}^m$ be the
set of vertices adjacent to $r$. Let $T_i$ be the subtree of $T$ with
vertices $r, v_i$, and the descendants of $v_i$ in $T$ 
(see Figure~1).
Let $f: T \to Y \subset (\Z,c_1)$ be a shy function. Then $f$ is constant on all but at most
2 of the $T_i$.
\end{exl}

\begin{proof} The assertion follows from Theorem~\ref{disconnections}.
\end{proof}

\begin{figure}
\label{tree}
\includegraphics[height=3in]{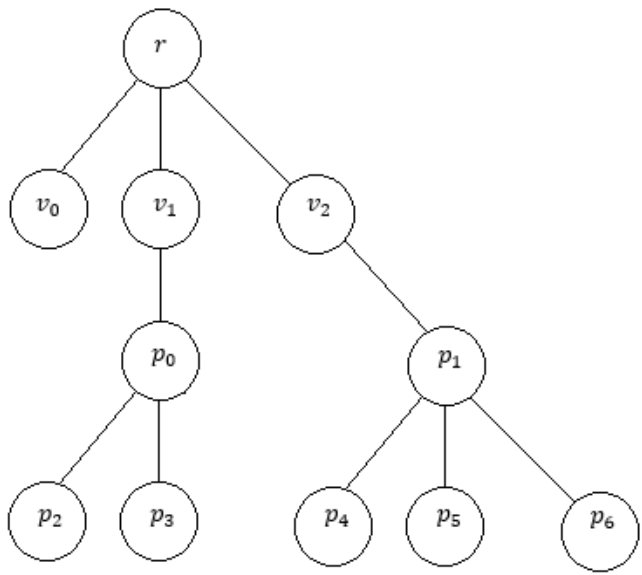}
\caption{A tree $T$ to illustrate
Example~\ref{tree-ex}. The vertex sets of $T_0$, $T_1$, and $T_2$ are,
respectively, $\{r, v_0\}$, $\{r,v_1,p_0, p_2, p_3\}$, and
$\{r,v_2, p_1, p_4, p_5, p_6\}$. A shy map from $T$ to a subset of
$(\Z,c_1)$ is non-constant on at most 2 of $T_0$, $T_1$, and $T_2$.
}
\end{figure}

\section{Further remarks}
We have made several contributions to our knowledge of digital
shy maps. In section~\ref{moreProps}, we studied the relations between
shy maps and both constant functions and isomorphisms.
In section~\ref{ops}, we showed
that shyness is preserved by compositions,
certain Cartesian products, and wedges. In section~\ref{intoZ},
we demonstrated several restrictions on shy maps onto subsets of $(\Z,c_1)$.

This research did not receive any specific grant from funding agencies in the 
public, commercial, or not-for-profit sectors.


\begin{thebibliography}{11}

\bibitem{Berge}
C. Berge,
{\em Graphs and Hypergraphs}, 2nd edition, North-Holland, Amsterdam, 1976.

%\bibitem{Borsuk}
%K. Borsuk,
%{\em Theory of Retracts},
%Polish Scientific Publishers, Warsaw, 1967.

\bibitem{Boxer94}
L. Boxer,
Digitally Continuous Functions,
{\em Pattern Recognition Letters} 15 (1994), 833-839.

\bibitem{Boxer99}
L. Boxer,
A Classical Construction for the Digital Fundamental Group,
{\em Pattern Recognition Letters} 10 (1999), 51-62.

\bibitem{Boxer05}
L. Boxer,
Properties of Digital Homotopy,
{\em Journal of Mathematical Imaging and Vision} 22 (2005),
19-26.

\bibitem{Boxer06}
L. Boxer,
Digital Products, Wedges, and Covering Spaces,
{\em Journal of Mathematical Imaging and Vision} 25 (2006), 159-171.

\bibitem{Boxer14}
L. Boxer,
Remarks on Digitally Continuous Multivalued Functions,
{\em Journal of Advances in Mathematics}
9 (1) (2014), 1755-1762.

\bibitem{BoxKar12}
L. Boxer and I. Karaca,
Fundamental Groups for Digital Products,
{\em Advances and Applications in Mathematical Sciences}
11(4) (2012), 161-180.


\bibitem{BoxS16}
L. Boxer and P.C. Staecker,
Connectivity Preserving Multivalued Functions in Digital Topology,
{\em Journal of Mathematical Imaging and Vision} 55 (3) (2016), 370-377.

\bibitem {egs08}
C. Escribano, A. Giraldo, and M. Sastre,
``Digitally Continuous Multivalued Functions,''
in \emph{Discrete Geometry for Computer Imagery}, Lecture Notes in 
Computer Science, v. 4992, Springer,
2008, 81-92.

\bibitem{egs12} 
C. Escribano, A. Giraldo, and M. Sastre,
``Digitally Continuous Multivalued Functions, Morphological 
Operations and Thinning Algorithms,''
\emph{Journal of Mathematical Imaging and Vision} 42 (2012), 76-91.

\bibitem{gs15}
A. Giraldo and M. Sastre,
On the Composition of Digitally Continuous Multivalued Functions,
{\em Journal of Mathematical Imaging and Vision}, 58 (2015), 
196-209.

\bibitem{Han05}
S.-E. Han,
Non-product property of the digital fundamental group,
{\em Information Sciences} 171 (2005), 73–-91.

%\bibitem{KR}
%T.Y. Kong and A. Rosenfeld, eds.
%{\em Topological Algorithms for Digital
%Image Processing}, Elsevier, 1996.

\bibitem{Kov}
V.A. Kovalevsky, A New Concept for Digital Geometry, 
{\em Shape in Picture}, Springer, New York (1994)

\bibitem{Rosenfeld}
A. Rosenfeld,
`Continuous' Functions on Digital Images,
{\em Pattern Recognition Letters} 4 (1987), 177-184.

\bibitem{Ts}
R. Tsaur and M. Smyth, "Continuous" multifunctions in discrete
spaces with applications to fixed point theory. In: Bertrand, G.,
Imiya, A., Klette, R. (eds.), {\em Digital and Image Geometry},
Lecture Notes in Computer Science, vol. 2243, pp. 151–162, Springer,
Berlin (2001), doi: 10.1007/3-540-45576-05

\end{thebibliography}
\end{document}